\documentclass[10pt]{article} 

\usepackage[utf8]{inputenc} 

\usepackage{geometry} 
\usepackage{graphicx} 
\usepackage{amsmath} 
\usepackage{amsfonts} 
\usepackage{amsthm} 
\usepackage{amssymb} 
\usepackage{enumitem}

\usepackage{algorithm}
\usepackage[noend]{algpseudocode}
\usepackage{authblk}
\usepackage{eucal}
\usepackage{url}
\usepackage{csvsimple}
\usepackage{sectsty}
\usepackage{verbatim}
\usepackage{longtable}

\sectionfont{\fontsize{12}{15}\selectfont}

\newtheorem{thm}{Theorem}[section]
\newtheorem{lem}[thm]{Lemma}

\newtheorem{cor}[thm]{Corollary}
\newtheorem{conj}[thm]{Conjecture}
\newtheorem{example}[thm]{Example}

\newtheorem{rem}[thm]{Remark}

\title{The Maximum Length and Isomorphism of Circuit Codes with Long Bit Runs}
\author{Kevin M. Byrnes
\thanks{E-mail:\texttt{dr.kevin.byrnes@gmail.com}}
}

\begin{document}
\maketitle

\begin{abstract}
Recently, Byrnes \cite{Byrnes2019} presented a formula for the maximum length of a symmetric circuit code that has a long bit run and odd spread.  
Here we show that the formula is also valid when the spread is even.  
We also establish that all maximum length symmetric circuit codes with long bit runs are isomorphic for
an infinite and nontrivial family of circuit codes, extending a previous result of Douglas \cite{Douglas2}.
\end{abstract}

\section{Introduction}
A $d$-dimensional circuit code of spread $k$  (also called a $(d,k)$ circuit code) is a cycle $C$ in $I(d)$, the graph of the $d$-dimensional hypercube, satisfying the distance property:
\begin{equation}
\label{eq1}
d_{I(d)}(x,y) \ge \min \{d_C(x,y),k\} \ \forall x,y\in C
\end{equation}
where $d_{I(d)}(x,y)$ denotes the minimum path length between vertices $x$ and $y$ in $I(d)$ (i.e. the Hamming distance) and $d_C(x,y)$ denotes the minimum path length between $x$ and $y$ in $C$.
Since $C$ is a subgraph of $I(d)$, $d_C(x,y)\ge d_{I(d)}(x,y) \ \forall x,y\in C$, so (\ref{eq1}) is equivalent to the pair of properties $\forall x,y \in C$: 
\begin{center}
\begin{minipage}{.7\textwidth}
\begin{itemize}
\setlength{\itemsep}{0pt}
\setlength{\parskip}{0pt}
\item [] $d_C(x,y)\le k \implies d_{I(d)}(x,y)=d_C(x,y), \text{ and }$
\item [] $d_C(x,y)\ge k \implies d_{I(d)}(x,y)\ge k.$ 
\end{itemize}
\end{minipage}
\end{center}
Intuitively, $C$ ``preserves distances'' up to $k$.
Circuit codes were first introduced in \cite{Kautz} as a type of error-correcting code, but have since been employed in a variety of applications, including:
rank-modulation schemes for flash memory \cite{Yehezkeally2012, Holroyd2017, Wang2017},
constructing worst-case examples for the analysis of combinatorial algorithms \cite{ByrnesDAM},
and analyzing the behavior of models for gene regulatory networks \cite{Zinovik2010}.
Circuit codes have been extensively studied in the case $k=2$, where they are called `snakes in the box' or `coils in the box' 
\cite{DanzerKlee, AK, JW, HS, Zemor, EL, vanZanten2001, Meyerson, Palombo, Allison}.
The general case where $k\in \mathbb{N}$ is surveyed in chapter 17 of \cite{Grunbaum} and studied in detail in 
\cite{Singleton, Klee, Douglas2, PN, Paterson, Hiltgen, Chebiryak, Hood, Byrnes, Byrnes2019}.

Vertices of $I(d)$ can be considered as vectors in $\{0,1\}^d$.  With the usual convention that every circuit code $C=(x_1,\ldots,x_N)$ of $I(d)$ begins with $x_1=\vec{0}$, a circuit $C$ is equivalent to its \emph{transition sequence} $T(C)=(\tau_1,\ldots,\tau_N)$ where $\tau_i$ denotes the single position in which $x_i$ and $x_{i+1}$ differ (with $x_{N+1}=x_1$).  The \emph{transition elements} of $T(C)$ are the unique members of $\{\tau_1,\ldots,\tau_N\}$, without loss of generality these are $\{1,\ldots,d\}$.  
A \emph{segment} of $T(C)$ is a cyclically consecutive subsequence $\omega=(\tau_i,\tau_{i+1},\ldots,\tau_j)$ (where subscripts $>N$ are reduced modulo $N$), 
and a segment $\omega$ is a \emph{bit run} if no transition element appears more than once.
E.g. $\omega=(1,2,3,4)$ is a bit run whereas $\omega=(1,2,3,1)$ is not.

Let $\varphi(C)$ denote the maximum length of a bit run in $T(C)$ , in general it is not true that every segment $\omega$ of $T(C)$ having length $|\omega|=\varphi(C)$ is a bit run.  Singleton \cite{Singleton} established the following formula relating the dimension $d$ of a circuit code $C$ and $\varphi(C)$.

\begin{thm}[\cite{Singleton} Theorem 3]
\label{dimension_lb}
Let $C$ be a $(d,k)$ circuit code with $|C|>2d$ and $\varphi(C)=k+l$ for $l\ge 2$.  Then $d\ge k+1+\lfloor \frac{k+l}{2}\rfloor$.
\end{thm}

\noindent 
When $k$ and $l$ have opposite parity this lower bound simplifies to $d\ge \frac{1}{2}(3k+l+1)$.
Generalizable constructions exist where this inequality is tight \cite{Douglas2, Byrnes2019}, and we say ``$C$ has a long bit run'' to denote such instances.  
In addition, (\cite{Singleton} Theorem 1) also established that for any $(d,k)$ circuit code $C$ with $|C|>2(k+1)$, $\varphi(C)\ge k+2$.  

Let $K(d,k)$ denote the maximum length of a $(d,k)$ circuit code.  Clearly $K(d,k)$ is even (as we consider circuits), and it is easy to show that $K(d,k)\ge 2k$ when $d\ge k$.
Furthermore, it is known that $K(d,1)=2^d$, $K(d,2)\ge \frac{3}{10} 2^d$ \cite{AK}, and $K(d,k) = 2^{d-\Theta(\log^2 d)}$ as $d\to \infty$ \cite{ByrnesSpinu}. 
However, no general formula for $K(d,k)$ is known, and computing the exact value of $K(d,k)$ for specific $(d,k)$ pairs is a challenging problem \cite{Kochut, Ostergard}.  
In a notable exception to this rule, Douglas \cite{Douglas2} (building upon previous results of \cite{Singleton}) established exact formulas for $K(\lfloor \frac{3k}{2} \rfloor+2,k)$ for $k$ even or odd, and $K(\lfloor \frac{3k}{2}\rfloor +3,k)$ for $k$ odd and $\ge 9$.  In addition, he proved that when $k$ is odd, any two maximum length $(\lfloor \frac{3k}{2}\rfloor +2,k)$ circuit codes are isomorphic (\cite{Douglas2} Theorem 4).
Recently Byrnes \cite{Byrnes2019} observed that in each of the preceding cases where a formula for $K(d,k)$ was known, it was 
possible to construct a maximum length circuit code $C$ that possessed a long bit run and which was \emph{symmetric}, i.e. $\tau_i = \tau_{i+N/2}$ for $i=1,\ldots, N/2$ where $T(C)=(\tau_1,\ldots,\tau_N)$.  
Let $\mathcal{F}(d,k,r)$ denote the family of $(d,k)$ circuit codes having maximum bit run length $\varphi(C)\ge r$, let $K(d,k,r)= {\max \{|C| \mid C\in \mathcal{F}(d,k,r)\}}$, and let $S(d,k,r)={\max \{|C| \mid C\in \mathcal{F}(d,k,r) \text{ and } C \text{ is symmetric}\}}$.  
By assuming both symmetry and the existence of a long bit run, \cite{Byrnes2019} was able to extend the family of $(d,k)$ pairs for which a formula for the maximum length was known, proving the following.

\begin{thm}[\cite{Byrnes2019} Theorem 3.9]
\label{byrnes1}
Let $k$ be odd, let $l$ be even and $\ge 2$ with $k\ge 2l+1$, and let $d=\frac{1}{2}(3k+l+1)$.  
Then $S(d,k,k+l)=4k+2l$.
\end{thm}

In this article we extend the formula for $S(d,k,k+l)$ of Theorem \ref{byrnes1} to circuit codes of even spread $k$.
Furthermore, using the same conditions of symmetry and the existence of a long bit run, 
we identify a range of parameter values of $l$ for which there is a unique (up to isomorphism) maximum length symetric 
$(\frac{1}{2}(3k+l+1),k,k+l)$ circuit code, thereby extending Theorem 4 of \cite{Douglas2}
\footnote{When $k$ is odd and $l=2$, part $(ii)$ of Theorem \ref{byrnes2} equates to the prior isomorphism result of \cite{Douglas2}.}.
Our results are summarized in the following theorem.

\begin{thm}
\label{byrnes2}
Let $k$ and $l$ be integers $\ge 2$ of opposite parity with $k\ge 2l+1$ if $k$ is odd, and $k\ge 2l-2$ if $k$ is even, and let $d=\frac{1}{2}(3k+l+1)$.  Then:
(i) $S(d,k,k+l)=4k+2l$, and
(ii) for $l=2 \text{ or } 3$ there is a unique (up to isomorphism) symmetric circuit code in $\mathcal{F}(d,k,k+l)$ of length $S(d,k,k+l)$.
\end{thm}

\noindent Here, by isomorphism, we mean that circuit codes $C$ and $C'$ of $I(d)$ with respective transition sequences $T(C)$ and $T(C')$
are isomorphic if $T(C)$ can be transformed into $T(C')$ by applying only a cyclic shift of transitions and an automorphism of $[d]$.
Throughout this article we will use the phrases ``$C$ is isomorphic to $C'$'' and ``$T(C)$ is isomorphic to $T(C')$'' interchangably.

Taken together, the conclusions of Theorem \ref{byrnes2} establish that the symmetric circuit codes in $\mathcal{F}(d,k,k+l)$ (for $d,k,\text{and } l$ as defined in the theorem)
represent one of the few non-trivial families of circuit codes that are well-understood, having both an exact formula for the maximum length and a range of parameter values where the maximum length circuit code is unique up to isomorphism.
Section 2 establishes the formula $S(d,k,k+l)=4k+2l$ when $k$ is even,
and in Section 3 we prove that any two symmetric codes in $\mathcal{F}(d,k,k+l)$ of length $S(d,k,k+l)$ are isomorphic for $l\in\{2,3\}$.
In Section 4 we conclude and present an example showing that the range of values for $l$ in part (ii) of Theorem \ref{byrnes2} cannot be extended when $k$ is even.
For ease of notation we will denote an arbitrary indexed set $\{\alpha_1,\alpha_2,\ldots,\alpha_m\}$ by $[\alpha_m]$ (or by $[m]$ in the case $\{1,\ldots,m\}$), and more generally we denote $\{\alpha_i,\alpha_{i+1},\ldots,\alpha_j\}$ by $[\alpha_i,\alpha_j]$
(or by $[i,j]$ in the case $\{i,i+1,\ldots,j\}$, e.g. $[k+l+1,d] = \{k+l+1,k+l+2,\ldots,d\}$).

\section{ The Maximum Length of Symmetric Circuit Codes with Long Bit Runs}
Let $C=(x_1,\ldots,x_N)$ be a $(d,k)$ circuit code with transition sequence $T(C)=(\tau_1,\ldots,\tau_N)$.
For a segment $\omega$ of $T(C)$, let $\delta(\omega)$ denote the number of transition elements that appear an odd number of times in $\omega$.
So $\omega$ is a bit run if $\delta(\omega)=|\omega|$, and $\delta(T(C))=0$ as $C$ is a circuit.
Observe that for $\omega=(\tau_i,\ldots,\tau_{j-1})$, $\delta(\omega)$ is equal to the number of positions in which $x_i$ and $x_j \in C$ differ, and thus equals $d_{I(d)}(x_i,x_j)$.
Also, $\delta(\omega,\tau_j)=d_{I(d)}(x_i,x_{j+1})=d_{I(d)}(x_i,x_j) \pm 1$ as $x_j$ and $x_{j+1}$ are adjacent in $I(d)$.  Therefore:
\begin{equation}
\label{eq2}
\text{for } \omega=(\tau_i,\ldots,\tau_{j-1}), \delta(\omega,\tau_j) = \delta(\omega) \pm 1.
\end{equation}
\noindent For a segment $\omega=(\tau_i,\ldots,\tau_{j-1})$ of $T(C)$ define its \emph{complement}, $\omega^{\complement}$, as
$\omega^{\complement}=(\tau_j,\ldots,\tau_{i-1})$.  Since $C$ is a cycle it follows that
$\delta(\omega)=\delta(\omega^{\complement})$ and $d_C(x_i,x_j)=\min\{|\omega|,|\omega^{\complement}|\}$.
Furthermore, when $|C|=N>2k$ it follows from (\ref{eq1}) that:
\begin{equation}
\label{eq3}
\text{if } |\omega| \le k+1, \text{then } \delta(\omega)=|\omega|, \text{and }
\end{equation}
\begin{equation}
\label{eq4}
\text{if } k\le |\omega|\le N-k, \text{then } \delta(\omega)\ge k.
\end{equation}
\noindent Here the equality in (\ref{eq3}) when $|\omega|=k+1$ follows as a result of (\ref{eq1}) and (\ref{eq2}).  Now we recall a useful property of circuit codes with long bit runs.

\begin{lem}[\cite{Byrnes2019} Lemma 3.5]
\label{lem1}
Let $C\in \mathcal{F}(d,k,k+l)$ where $k$ and $l$ are integers of opposite parity, $d=\frac{1}{2}(3k+l+1)$, and having length $|C|>2d$.  Then any segment $\omega=(\tau_{r+1},\ldots,\tau_{r+2k+l})$ of $T(C)$ ($r\ge 0$) where $(\tau_{r+1},\ldots,\tau_{r+k+l})$ is a bit run has $\delta(\omega)=k+1$, $\delta(\omega,\tau_{r+2k+l+1})=k$, and all $d$ transition elements appear once or twice in $\omega$.  Also, $n_1=k$ elements appear exactly once in $(\omega,\tau_{r+2k+l+1})$ and $n_2=d-n_1=\frac{k+l+1}{2}$ elements appear twice in $(\omega,\tau_{r+2k+l+1})$.
\end{lem}

\begin{rem}
\label{rem1}
Since $\delta(\omega)=k+1$ and $\delta(\omega,\tau_{r+2k+l+1})=k$, Lemma \ref{lem1} implies $\tau_{r+2k+l+1}$ appears once in $\omega$.  Therefore $n_1+1$ elements appear exactly once in $\omega$, and $n_2-1$ elements appear twice in $\omega$.
\end{rem}

Now suppose that $k$ and $l$ are integers of opposite parity with $d=\frac{1}{2}(3k+l+1)$, and let $C\in \mathcal{F}(d,k,k+l)$ with $T(C)=(\tau_1,\ldots,\tau_N)$.  Consider the segment $\omega=(\tau_1,\tau_2,\ldots,\tau_{2k+l})$, without loss of generality we may assume that $T(C)$ begins with a bit run of length $k+l$, so:
\begin{equation}
\label{eq5}
\omega=(\underbrace{1,2,\ldots,k+l}_{\omega_1},\underbrace{\tau_{k+l+1},\ldots,\tau_{2k+l}}_{\omega_2}).
\end{equation}
Observe that $\omega_1$ is a bit run of length $k+l$ by construction, and $\omega_2$ is also a bit run (of length $k$) by (\ref{eq3}).
By Lemma \ref{lem1}, all $d$ transition elements appear once or twice in $\omega$.  Let $S_1$ denote the set of transition elements appearing once in $\omega$, and let $S_2$ denote the set of transition elements appearing twice in $\omega$.  The next result shows that the $k+1$ transitions in $T(C)$ following $\omega$ alternate between $S_1$ and $S_2$.

\begin{lem}
\label{lem2}
Let $k$ and $l$ be integers of opposite parity and let $d=\frac{1}{2}(3k+l+1)$.  Let $C\in \mathcal{F}(d,k,k+l)$ with $|C|\ge 4k+l+1$ and suppose the first $2k+l$ transitions of $T(C)$ are $(\omega_1,\omega_2)$ where $\omega_1$ is a bit run of length $k+l$.  Let $S_j$ denote the set of transition elements appearing $j$ times in $(\omega_1, \omega_2)$ for $j=1,2$.  Then the segment $(\omega_1,\omega_2,\alpha_1,\alpha_2,\ldots,\alpha_{k+1})$ of $T(C)$ has $\alpha_i\in S_1\cap \omega_1$ if $i$ odd and $\alpha_i\in S_2$ if $i$ even.
\end{lem}

\begin{proof}
We demonstrate a slightly stronger claim:  $\alpha_i\in S_1\cap \omega_1$ and $\delta(\omega_1,\omega_2,\alpha_1,\ldots,\alpha_i)=k$ if $i$ is odd, and $\alpha_i\in S_2$ and $\delta(\omega_2,\alpha_1,\ldots,\alpha_i)=k$ if $i$ is even.  As $\omega_1$ and $\omega_2$ are bit runs, any transition element $t\in S_2$ appears exactly once in $\omega_1$ and exactly once in $\omega_2$.

Let $i=1$, then by Lemma \ref{lem1}, $\delta(\omega_1,\omega_2,\alpha_1)=k$ and $\delta(\omega_1,\omega_2)=k+1$.  Thus $\alpha_1$ occurs an odd number of times in $(\omega_1,\omega_2)$, so $\alpha_1\in S_1$.  Since $|(\omega_2,\alpha_1)|=k+1$, from (\ref{eq3}) we have $\delta(\omega_2,\alpha_1)=k+1$ and so $\alpha_1\not\in \omega_2$.  Thus $\alpha_1\in S_1\cap \omega_1$.  Next, let $i=2$.  Since $\delta(\omega_1,\omega_2,\alpha_1)=k$, $\delta(\omega_1,\omega_2,\alpha_1,\alpha_2)=k+1$ following from (\ref{eq2}) and (\ref{eq4}), and thus $\alpha_2$ occurs an even number of times in $(\omega_1,\omega_2,\alpha_1)$.  Clearly $\alpha_2\neq \alpha_1$, so $\alpha_2$ occurs an even number of times in $(\omega_1,\omega_2)$, thus $\alpha_2\in S_2$.  
Finally, since $\delta(\omega_2,\alpha_1)=k+1$, we have $\delta(\omega_2,\alpha_1,\alpha_2)\in \{k,k+2\}$ by (\ref{eq2}).  
As $\alpha_2\in S_2$, $\alpha_1\not\in\omega_2$, and all transitions in $\omega_2$ are distinct, we have $\delta(\omega_2,\alpha_1,\alpha_2)=k$.

Now suppose that the claim holds $\forall i\le m$, where $2\le m \le k$, and let $i=m+1$.  First suppose that $m$ is even.  By the inductive hypothesis $\delta(\omega_2,\alpha_1,\ldots,\alpha_m)=k$, and so $\delta(\omega_2,\alpha_1,\ldots,\alpha_{m+1})=k+1$ by (\ref{eq2}) and (\ref{eq4}).  As $\alpha_{m+1}\not\in[\alpha_{m}]$ (since $(\alpha_1,\ldots,\alpha_{m+1})$ is a segment of length $m+1\le k+1$) this implies that $\alpha_{m+1}$ occurs an even (possibly 0) number of times in $\omega_2$.  But since $\omega_2$ is a bit run this means $\alpha_{m+1}$ occurs 0 times, so $\alpha_{m+1}\in S_1\cap \omega_1$.  
Again using the inductive hypothesis, we have $\delta(\omega_1,\omega_2,\alpha_1,\ldots,\alpha_{m-1})=k$ and $\alpha_m\in S_2$, so $\delta(\omega_1,\omega_2,\alpha_1,\ldots,\alpha_m)=k+1$.  
Since $\alpha_{m+1}\in S_1\cap \omega_1$ and not in $[\alpha_m]$ this gives us $\delta(\omega_1,\omega_2,\alpha_1,\ldots,\alpha_{m+1})=k$.

Suppose that $m$ is odd.  By the induction hypothesis, $\delta(\omega_1,\omega_2,\alpha_1,\ldots,\alpha_m)=k$, so by (\ref{eq2}) and (\ref{eq4}) $\delta(\omega_1,\omega_2,\alpha_1,\ldots,\alpha_{m+1})=k+1$, and $\alpha_{m+1}$ must occur an even number of times in $(\omega_1,\omega_2,\alpha_1,\ldots,\alpha_m)$.  As $\alpha_{m+1}\not\in[\alpha_m]$ this implies $\alpha_{m+1}\in S_2$.  
Finally, by the inductive hypothesis we have $\delta(\omega_2,\alpha_1,\ldots,\alpha_{m-1})=k$ and $\alpha_m\in S_1\cap \omega_1$,
and since $\alpha_m\not\in [\alpha_{m-1}]$
this further implies $\delta(\omega_2,\alpha_1,\ldots,\alpha_m)=k+1$.  
Hence $\delta(\omega_2,\alpha_1,\ldots,\alpha_{m+1})\in \{k,k+2\}$ by (\ref{eq2}).
Since $\alpha_{m+1}\in S_2$, $\alpha_{m+1}\not\in [\alpha_m]$, and $\omega_2$ is a bit run, this implies $\alpha_{m+1}$ occurs exactly once in $(\omega_2,\alpha_1,\ldots,\alpha_m)$, so $\delta(\omega_2,\alpha_1,\ldots,\alpha_{m+1})=k$.

Thus we have proven the claim $\forall i\le k+1$.
\end{proof}

\begin{cor}
\label{cor1}
Let $k$ be even, let $l$ be odd and $\ge 3$ with $k\ge 2l-2$, and let $d=\frac{1}{2}(3k+l+1)$.  Then $S(d,k,k+l)=4k+2l$.
\end{cor}
\begin{proof}
From \cite{Byrnes2019} (Lemma 3.6 and Corollary 4.3) we have $S(d,k,k+l)\in {\{4k+2l,4k+2l+2\}}$.
Assume for contradiction that there exists symmetric $C\in \mathcal{F}(d,k,k+l)$ with $|C|=4k+2l+2$.
Without loss of generality
\begin{equation*}
T(C)=(\underbrace{1,2,\ldots,k+l}_{\omega_1},\underbrace{\beta_1,\ldots,\beta_k}_{\omega_2},\alpha, 1,2, \ldots,k+l,\beta_1,\ldots,\beta_k,\alpha).
\end{equation*}
\noindent Observe that as $(\alpha,1,2)$ is a segment of $T(C)$ of length $3\le k+1$, by (\ref{eq3}) $\alpha\neq 2$.
Now consider the segment $(\alpha,1,2,\ldots,k)$ of $T(C)$, and define $S_1$ and $S_2$ as in Lemma \ref{lem2}.
It follows that $\alpha\in S_1\cap \omega_1$, and for $1\le i \le k$ that $i\in S_1\cap \omega_1$ if $2\mid i$ and $i\in S_2$ if $2\nmid i$.
Next, consider the segment $\omega=(3,4,\ldots,k+l,\omega_2,\alpha,1)$.
Rearranging the order of the sequence $\omega$ does not change $\delta(\omega)$, so $\delta(\omega)=(1,3,4,\ldots,k+l,\omega_2,\alpha)=\delta(\omega_1,\omega_2,\alpha)-1$ since 2 occurs only once in $(\omega_1,\omega_2,\alpha)$ (specifically $2\in S_1\cap \omega_1$ and $\alpha \neq 2$).
By Lemma \ref{lem1}, $\delta(\omega_1,\omega_2,\alpha)=k$, so $\delta(\omega)=k-1$, but as $k\le |\omega|\le N-k$ this contradicts (\ref{eq4}), hence no such $C$ exists.
\end{proof}

\section{Isomorphism of Maximum Length Circuit Codes}
Throughout this section we assume that $k$ and $l$ are integers $\ge 2$ of opposite parity with: $l\ge 2$, $k\ge 2l-2$ if $2\mid k$ and $k\ge 2l+1$ if $2 \nmid k$, and let $d=\frac{1}{2}(3k+l+1)$.
By Theorem \ref{byrnes1} and Corollary \ref{cor1}, any maximum length symmetric $C\in \mathcal{F}(d,k,k+l)$ has $|C|=4k+2l$, and we may assume $T(C)$ has the form:

\begin{equation}
\label{eq6}
T(C)=(\underbrace{1,2,\ldots,k+l}_{\omega_1},\underbrace{\beta_1,\beta_2,\ldots,\beta_k}_{\omega_2},\omega_1,\omega_2).
\end{equation}
\noindent The segment $\omega_2$ in (\ref{eq6}) corresponds to the segment $\omega_2$ from (\ref{eq5}), and is a bit run by (\ref{eq3}) as previously observed.

Now consider the assumption:
\begin{equation}
\label{eq7}
\beta_i=2i \text{ for } 1\le i\le l-1.
\end{equation}

\noindent We will show (Lemma \ref{lem3_2}) that assumption (\ref{eq7}) holds when $l=2$ or $3$.  Furthermore, we establish (Lemma \ref{lem3_4}) that any maximum length $C\in \mathcal{F}(d,k,k+l)$ satisfying assumption (\ref{eq7}) must have $T(C)$ isomorphic to:

\begin{equation}
\label{eq8}
(\underbrace{1,2,\ldots,k+l}_{\omega_1},\underbrace{2,4,\ldots,2(l-1)}_{\omega_2^A},\underbrace{\gamma_1,\epsilon_1,\gamma_2,\epsilon_2,\ldots,\gamma_m,\epsilon_m}_{\omega_2^B},\omega_1,\omega_2^A,\omega_2^B)
\end{equation}
where $m=\frac{1}{2}(k-l+1)$, $\gamma_i=k+l+i$ and $\epsilon_i=2(l-1+i)$ for $1\le i\le m$, thus proving part (ii) of Theorem \ref{byrnes2}.

\begin{lem}
\label{lem3_1}
Consider the segments of $T(C): \rho_1=(r,r+1,\ldots,k+l,\beta_1,\ldots,\beta_k)$ for $1\le r\le k+l$, and 
$\rho_2=(\beta_1,\ldots,\beta_k,1,\ldots,s)$ for $1\le s \le k+l$, and let 
$n_1'=|[r,k+l]\cap [\beta_k]|$ and $n_2'=|[s]\cap [\beta_k]|$.  
Then $n_1'\le \lfloor \frac{1}{2}(k+l-r+1)\rfloor$ and $n_2'\le \lfloor \frac{s}{2} \rfloor$.
\end{lem}
\begin{proof}
Observe that $|\rho_1|=2k+l-r+1$ and $|\rho_2|=k+s$, hence by (\ref{eq4}) we have $\delta(\rho_1)\ge k$ and $\delta(\rho_2)\ge k$.  As no transition is repeated within $\omega_1$ nor $\omega_2$, this implies $k\le \delta(\rho_1)=2k+l-r+1-2n_1'$ and $k\le \delta(\rho_2)=k+s-2n_2'$.  Hence $2n_1'\le k+l-r+1$ and $2n_2'\le s$.
\end{proof}

\begin{lem}
\label{lem3_2}
Assumption (\ref{eq7}) is valid for $l=2 \text{ or } 3$.
\end{lem}

\begin{proof}
Let $S_1$ and $S_2$ be defined as in the statement of Lemma \ref{lem2}, and let $r,s,n_1'$, and $n_2'$ be defined as in the statement of Lemma \ref{lem3_1}.  We begin with some elementary observations: for $i\in \{1,\ldots,k\}$ the segment $\psi_i=(\beta_i,\ldots,\beta_k,1,\ldots,i)$ of $T(C)$ has length $k+1$ and so by (\ref{eq3}) is a bit run;
the $k+1$ transitions in $T(C)$ following $(\omega_1,\omega_2)$ are $(1,2,\ldots,k+1)$, so by 
Lemma \ref{lem2} we have $2,4,\ldots,2\lfloor \frac{k}{2}\rfloor \in S_2$ (and thus in $\omega_2$), and $1,2,\ldots,2 \lceil \frac{k}{2} \rceil -1 \in S_1 \cap \omega_1$ (and thus not in $\omega_2$).
Thus we see that $2\in \omega_2$, but $2\not\in (\beta_2,\ldots,\beta_k)$ (since $\psi_2$ is a bit run), hence $\beta_1=2$.
By similar logic, we deduce either $\beta_2=4$ or $\beta_3=4$.
At this point we have established that assumption (\ref{eq7}) holds when $l=2$ or if $l=3$ and $\beta_2=4$, thus we may assume that $l=3$ and $\beta_3=4$, giving us 
$(\omega_1,\omega_2)=(1,2,\ldots,k+3,2,\beta_2,4,\beta_4,\ldots,\beta_k)$.

The segment $\omega=(2,\ldots,k+3,2,\beta_2,4)$ has $|\omega|=k+5$, so by (\ref{eq4}), $\delta(\omega)\ge k$.
Since ${(2,\ldots,k+3)}\subseteq \omega_1$ and $(2,\beta_2,4)\subseteq \omega_2$ are bit runs, no transition element in $\omega$ occurs more than twice, 
which implies $\delta(\omega)=|\omega|-2\cdot|[2,k+3]\cap (2,\beta_2,4)|$.
Thus $\delta(\omega)\le k+1$ as $2 \text{ and }4 \in [2,k+3]$.
Hence $\beta_2\not\in [2,k+3]$ (else $\delta(\omega)<k$), and since $\psi_1$ is a bit run, $\beta_2 \neq 1$, so without loss of generality $\beta_2=k+4$.
We proceed by showing that the pattern $\beta_i=i+1$ if $i$ odd and $\beta_i=k+3+\frac{i}{2}$ if $i$ even persists until some $\beta_i$ and $\beta_{i+1}$ are consecutive even elements of $[k+3]$, and then we use this structure to show that 
$T(C)$ is isomorphic to some transition sequence $T'$ satisfying (\ref{eq6}) and (\ref{eq7}).
By this, we mean $T'=(1,2,\ldots,k+3,\beta_1',\ldots,\beta_k',1,2,\ldots,k+3,\beta_1',\ldots,\beta_k')$
where: $(\beta_1',\ldots,\beta_k')$ is a bit run, and $\beta_1'=2$ and $\beta_2'=4$.
Note that since isomorphism preserves symmetry, length, and spread (so any segment of length $\le k+1$ is a bit run), it will be sufficient to demonstrate an isomorphic transition sequence $T'$ whose first $k+5$ transitions are $(1,2,\ldots,k+3,2,4)$.

Suppose there exists an even integer $t$, $t\le k-2$, such that $\forall i\le t: \beta_i=i+1$ if $i$ odd, and $\beta_i=k+3+\frac{i}{2}$ if $i$ even.  By the case assumption that $l=3$ and $\beta_3=4$, we have established this is true for $t=2$.  We claim $\beta_{t+1}=t+2$ and $\beta_{t+2}=t+4$ or $k+3+\frac{t+2}{2}$.  
Since $t$ is even and $t \le k-2$, $t+2\in \omega_2$ as previously observed.  
And, since $\psi_{t+2}$ is a bit run, the only place where $t+2$ fits into $\omega_2$ is $\beta_{t+1}=t+2$.
Now either $\beta_{t+2}\not\in[k+3]$ (in which case, without loss of generality, $\beta_{t+2}=k+3+\frac{t+2}{2}$ and we are done), or $\beta_{t+2}\in [k+3]$.  
By Remark \ref{rem1} we have $|[k+3]\cap [\beta_k]|=n_2-1=d-k-1=\frac{k+2}{2}$ (here we use the case assumption $l=3$).
Now suppose $\beta_{t+2}\in [k+3]$ and decompose $|[k+3]\cap [\beta_k]|$ into $|[t+2]\cap [\beta_k]| + |\{t+3,t+4\} \cap [\beta_k]| + |[t+5,k+3] \cap [\beta_k]|$.
Using the values we established for $\beta_1,\ldots,\beta_{t+1}$ and Lemma \ref{lem3_1} (with $s=t+2$) we have 
$|[t+2]\cap [\beta_k]|=\frac{t}{2}+1$, and by using Lemma \ref{lem3_1} with $r=t+5$ we have $|[t+5,k+3]\cap [\beta_k]|\le \lfloor \frac{1}{2}(k-1-t)\rfloor = \frac{1}{2}(k-t-2)$, thus $t+3$ or $t+4\in [\beta_k]$.
Using Lemma \ref{lem3_1} with $s=t+3$ yields $|[t+3]\cap [\beta_k]|\le \frac{t}{2}+1$, so $t+3\not\in [\beta_k]$ as 
$|[t+2]\cap [\beta_k]|=\frac{t}{2}+1$.  
Hence $t+4\in[\beta_k]$, and since $\psi_{t+4}$ is a bit run this means $t+4\in \{\beta_{t+2},\beta_{t+3}\}$.
Finally, since we presume $\beta_{t+2}\in [k+3]$ we must have $\beta_{t+2}=t+4$, otherwise 
$\omega=(2,3,\ldots,k+3,2,k+4,\ldots,t,k+3+\frac{t}{2},t+2,\beta_{t+2},t+4)=(2,\ldots,k+3,\beta_1,\ldots,\beta_{t+3})$ 
is a segment of $T(C)$ with $k\le |\omega|=k+t+5 \le N-k$
and $\delta(\omega)=|\omega|-2\cdot |[2,k+3]\cap [\beta_{t+3}]|=k+t+5-2(\frac{t}{2}+3)=k-1$, violating (\ref{eq4}).

As there are $d-(k+l)<k/2$ transition elements in $[k+l+1,d]$ it follows from the preceding paragraph that there is some minimum index $t$ ($t$ even) for which $\beta_{t+1}$ and $\beta_{t+2}$ are consecutive even elements of $[k+3]$, specifically $\beta_{t+1}=t+2$ and $\beta_{t+2}=t+4$.  Therefore the first $(k+3)+(t+2)$ transitions of $T(C)$ are:

\begin{equation}
\label{eq9}
(1,2,\ldots,k+3,2,k+4,4,k+5,\ldots,t,k+3+t/2,t+2,t+4).
\end{equation}
Also observe that the segment $\omega$ of $T(C)$ defined by:
\begin{equation}
\label{eq10}
\omega=(t+1,t+2,\ldots,k+3,2,k+4,\ldots,\beta_{t-1}=t,\beta_t=k+3+t/2)
\end{equation}
is a bit run and has length $k+3$.  
We now show that $T(C)$ satisfies (\ref{eq7}) by demonstrating an isomorphic transition sequence $T'$ for which (\ref{eq7}) clearly holds.
Recall that cyclically shifting $T(C)$ or applying an automorphism of $[d]$ to $T(C)$ results in a transition sequence $T'$ isomorphic to $T(C)$.
First, cyclically shift $T(C)$ to get a new transition sequence, $\hat{T}$, that begins with the segment $\omega$ from (\ref{eq10}).  
Following from (\ref{eq9}), the first $k+5$ transitions of $\hat{T}$ are: $\omega,t+2,t+4$.
Now apply to $\hat{T}$ any automorphism of $[d]$ that maps the $i$th transition of $\hat{T}$ to $i$ for $1\le i \le k+3$ (since $\omega$ is a bit run of length $k+3$ at least one such automorphism exists).
The resulting transition sequence $T'$ is isomorphic to $T(C)$, and the first $k+5$ transitions of $T'$ are: 
$1,2,\ldots,k+3,2,4$, hence $T'$ satisfies (\ref{eq6}) and  (\ref{eq7}).
\end{proof}

The following example illustrates the isomorphism transformation presented in the final paragraph of Lemma \ref{lem3_2}.

\begin{example}[A nontrivial isomorphism when $l=3$]
\label{ex1}
Consider the circuit codes $C$ and $C'\in \mathcal{F}(11,6,9)$
with transition sequences $T(C)=(\omega_A,\omega_A)$ and $T(C')=(\omega_B,\omega_B)$ where:
\small
\[
\omega_A=1,2,3,4,\underbrace{5,6,7,8,9,2,10,4,11}_{\omega},6,8
\normalsize
\text{ and }
\small
\omega_B=1,2,3,4,5,6,7,8,9,2,4,10,6,11,8
\]
\normalsize
\noindent and $\omega$ indicates the segment from (\ref{eq10}).
Note that $T(C')$ satisfies Lemma \ref{lem3_2} and that $T(C)$ is isomorphic to $T(C')$ via the automorphism: $\sigma(1)=10, \sigma(2)=6, \sigma(3)=11, \sigma(4)=8, \sigma(5)=1, \sigma(6)=2, \sigma(7)=3, \sigma(8)=4, \sigma(9)=5, \sigma(10)=7, \sigma(11)=9$.  If we had chosen the automorphism $\tau(i)=\sigma(i), i\not\in \{1,3\}$, $\tau(1)=11, \tau(3)=10$ we would also get a transition sequence satisfying Lemma \ref{lem3_2} (and also isomorphic to $T(C')$).
\end{example}

Recall, if $C$ is a symmetric circuit code in $\mathcal{F}(d,k,k+l)$ of length $S(d,k,k+l)=4k+2l$, then $T(C)$ has the form $(\omega_1,\omega_2,\omega_1,\omega_2)$ as in (\ref{eq6}).
If $T(C)$ also satisfies assumption (\ref{eq7}) then the first $l-1$ transitions of $\omega_2$ are 
$\omega_2^A=(2,4,\ldots,2(l-1))$ and therefore $T(C)$ has the form:
\begin{equation}
\label{eq11}
(\omega_1,\omega_2^A,\underbrace{\alpha_1,\alpha_2,\ldots,\alpha_{2m}}_{\omega_2^{B'}},\omega_1,\underbrace{\omega_2^A,\omega_2^{B'}}_{\omega_2})
\end{equation}
where $m=\frac{1}{2}(k-l+1)$ and $(\omega_2^A,\omega_2^{B'})=\omega_2$.
In the next lemma we show that $\omega_2^{B'}$ equals the segment $\omega_2^B$ defined in (\ref{eq8}), and thus $T(C)$ is isomorphic to (\ref{eq8}).

\begin{lem}
\label{lem3_4}
Let $C$ be a symmetric circuit code in $\mathcal{F}(d,k,k+l)$ having length $S(d,k,k+l)$ and suppose that $T(C)$ satisfies (\ref{eq7}).  Then $T(C)$ is isomorphic to (\ref{eq8}).
\end{lem}

\begin{proof}
As observed, for any such circuit code $C$, we may assume $T(C)$ has the form (\ref{eq11}).
We now show: 
(i) $\alpha_i\in [k+l+1,d]$ if $2\nmid i$ and $\alpha_i\in [k+l]$ if $2\mid i$, for $1\le i\le 2m$, and
(ii) $\alpha_{2p}=2(l-1+p)$ for $1\le p\le m$.
Together these will imply that $\omega_2^{B'}=\omega_2^B$ and so $T(C)$ is isomorphic to (\ref{eq8}).

To prove (i) we claim that $\alpha_1=k+l+1$ and that $\alpha_i$ and $\alpha_{i+1}$ alternate as members of $[k+l+1,d]$ and $[k+l]$ for $1\le i \le 2m-1$.
Following from Lemma \ref{lem1}, all of the $m\ {(=d-(k+l))}$ transition elements in $[k+l+1,d]$ must appear in $\omega_2$ and exactly once (as $\omega_2$ is a bit run).
Furthermore, since $(\omega_2^A,\omega_2^{B'})=\omega_2$, and $\omega_2^A\cap [k+l+1,d]=\varnothing$, all of these transition elements appear in $\omega_2^{B'}$.
Consider $\omega=(2,3,\ldots,k+l,\omega_2^A)$, then $\delta(\omega)=k$ so by (\ref{eq2}) and (\ref{eq4}), $\delta(\omega,\alpha_1)=k+1$.
This means $\alpha_1$ occurs an even number (either 0 or 2) times in $\omega$, implying: $\alpha_1\in {\{2,3,\ldots,k+l\}} \cap \omega_2^A$, impossible since $(\omega_2^A,\alpha_1)$ is a bit run by (\ref{eq3});
$\alpha_1=1$, impossible since $(\omega_2^{B'},1)$ is a bit run by (\ref{eq3});
or $\alpha_1\in [k+l+1,d]$.
Thus $\alpha_1\in [k+l+1,d]$ and without loss of generality, $\alpha_1=k+l+1$.

Assume for contradiction that there exists a minimum $i\in \{1,\ldots,2m-1\}$ for which $\alpha_i$ and $\alpha_{i+1}$ either both belong to $[k+l+1,d]$ or both belong to $[k+l]$.
If $2\nmid i$ then $\alpha_i$ and $\alpha_{i+1}\in [k+l+1,d]$.
Then $T(C)$ contains the segment $\omega$ of length $2k+l-2$, defined as 
$\omega=(\alpha_{i+2},\ldots,\alpha_{2m},\omega_1,\omega_2^A,\alpha_1,\ldots,\alpha_{i-1})$ if $i\le 2m-2$, and
$\omega=(\omega_1,\omega_2^A,\alpha_1,\ldots,\alpha_{2m-2})$ if $i=2m-1$.
In both cases, $\omega$ can be formed from $(\omega_1,\omega_2)$ by removing the transitions $\alpha_i$ and $\alpha_{i+1}$ and then rearranging the remaining transitions.
Since $\alpha_i,\alpha_{i+1}\in [k+l+1,d]$, each appears exactly once in $(\omega_1,\omega_2)$,
and therefore: $\delta(\omega)=\delta(\omega_1,\omega_2)-2=k-1$ (as $\delta(\omega_1,\omega_2)=k+1$ by Lemma \ref{lem1}), contradicting (\ref{eq4}).
If $2\mid i$ then $\alpha_i \text{ and } \alpha_{i+1}\in [k+l]$.
No transition element in $(\omega_1,\omega_2)$ occurs more than twice (by Lemma \ref{lem1}), and
because Lemma \ref{lem2} implies only even elements of $[k+1]$ occur twice in $(\omega_1,\omega_2)$ (note the usage of the notation $\alpha_j$ differs there), no $\alpha_j\in [k+l]$ is equal to 1.
Thus if $\alpha_j\in [k+l]$ then $\alpha_j$ occurs twice in $(2,3,\ldots,k+l,\omega_2^A,\omega_2^{B'})$, and if $\alpha_j\in [k+l+1,d]$ then $\alpha_j$ occurs once in $(2,3,\ldots,k+l,\omega_2^A,\omega_2^{B'})$.
For $j\in \{0,1,\ldots,i-1\}$ define $\psi_j=(2,3,\ldots,k+l,\omega_2^A,\alpha_1,\ldots,\alpha_{j+1})$.
By definition of $i$, for any $j\in [i-1]$, $\alpha_j$ and $\alpha_{j+1}$ alternate as members of $[k+l+1,d]$ and $[k+l]$, beginning with $\alpha_1=k+l+1$.
This implies $\delta(\psi_j)=k$ if $2\nmid j$, and $\delta(\psi_j)=k+1$ if $2\mid j$.
In particular, for $\psi_{i-1}=(2,3,\ldots,k+l,\omega_2^A,\alpha_1,\ldots,\alpha_i)$ we have $\delta(\psi_{i-1})=k$,
and since $\alpha_{i+1}\in [k+l]$ by the case assumption, $\alpha_{i+1}$ occurs twice in the segment  $(\psi_{i-1},\alpha_{i+1})$ of $T(C)$.
Thus by (\ref{eq2}), $\delta(\psi_{i-1},\alpha_{i+1})=\delta(\psi_{i-1})-1=k-1$, and again we contradict (\ref{eq4}).
Therefore no such index $i$ exists, proving (i).

We prove (ii) by induction.
Since $(\omega_2^A,\omega_2^{B'})$ is a bit run, all of $\{\alpha_1,\ldots,\alpha_{2m}\}$ are distinct and none 
of the attendant transition elements appear in
$\omega_2^A=(2,4,\ldots,2(l-1))$.
In particular, from $(i)$ we have $\alpha_i\in [k+l+1,d]$ if $i$ is odd and $\alpha_i\in [k+l]$ if $i$ is even.
For $p\in \{1,\ldots,m\}$ let $t(2p)$ denote the value in $[k+l]$ assumed by $\alpha_{2p}$.
Define the segment $\omega(2p)$ of $T(C)$ as:
$\omega(2p)=(t(2p),t(2p)+1,\ldots,k+l,\omega_2^A,\alpha_1,\ldots,\alpha_{2p})$,
then $|\omega(2p)|=k+2l+2p-t(2p)$.
Observe that $\omega(2p)$ has at least one repeated transition element (e.g. $t(2p)=\alpha_{2p}$), so (\ref{eq3}) implies $k+2\le |\omega(2p)|$,
hence $\alpha_{2p}=t(2p)\le 2(l-1+p) \ \forall p\in \{1,\ldots,m\}$.

If $p=1$, applying Lemma \ref{lem3_1} with $s=2l-1$ yields
$|(\omega_2^A,\omega_2^{B'})\cap [2l-1]|\le \lfloor \frac{2l-1}{2}\rfloor = l-1$.
Since $|\omega_2^A \cap [2l-1]|=l-1$ by construction, this implies
$|\omega_2^{B'}\cap [2l-1]|=0$, and hence $\alpha_{2p}\ge 2l$.
Thus $\alpha_2=2l$, establishing the base case.
Now suppose that $\alpha_{2p}=2(l-1+p) \ \forall p\le r$ for some $r\in \{1,\ldots,m-1\}$, and let $p=r+1$.
Then, applying Lemma \ref{lem3_1} with $s=2(l-1+p)-1$ yields
$|(\omega_2^A,\omega_2^{B'})\cap [2(l-1+p)-1]|\le l-2+p = l-1+r$.
By assumption, $|(\omega_2^A,\alpha_1,\ldots,\alpha_{2r})\cap [2(l-1+r)]|=l-1+r$,
therefore:  $|(\omega_2^A,\alpha_1,\ldots,\alpha_{2r})\cap [2(l-1+p)-1]|=l-1+r$ (following from the prior upper bound since
$[2(l-1+r)]\subseteq [2(l-1+p)-1]$), and
$|(\alpha_{2r+1},\ldots,\alpha_{2m})\cap [2(l-1+p)-1]|=0$,
so $\alpha_{2p}\ge 2(l-1+p)$.
Thus $\alpha_{2p}=2(l-1+p)$, completing the inductive step and the proof of $(ii)$.

Finally, by (i), $\alpha_i\in [k+l+1,d]$ for odd $i$.
Since each transition element in $[k+l+1,d]$ appears once in $(\omega_1,\omega_2)$ (specifically in $\omega_2^{B'}$), we may assume without loss of generality that 
$\alpha_i=k+l+\lceil \frac{i}{2} \rceil$ for odd $i\in \{1,\ldots,2m\}$.
By (ii), $\alpha_i=2(l-1+\frac{i}{2})$ for even $i\in \{1,\ldots,2m\}$.
Hence for $i\in \{1,\ldots,2m\}$: $\alpha_i=k+l+\lceil \frac{i}{2}\rceil=\gamma_{\lceil i/2 \rceil}$ if $i$ is odd,
and $\alpha_i=2(l-1+\frac{i}{2})=\epsilon_{i/2}$ if $i$ is even,
where $\gamma_j$ and $\epsilon_j$ are as defined in (\ref{eq8}).
Thus $\omega_2^{B'}=\omega_2^B$ and $T(C)$ is isomorphic to (\ref{eq8}).
\end{proof}

\begin{proof}[\textbf{Proof of Theorem \ref{byrnes2}}]
Claim (i) of the theorem is established for odd $k$ by Theorem \ref{byrnes1}, and for even $k$ by Corollary \ref{cor1}. 
Claim (ii) is established by Lemma \ref{lem3_2} and Lemma \ref{lem3_4}, which collectively prove that, for $l=2 \text{ or }3$, if $C$ is a symmetric circuit code in $\mathcal{F}(d,k,k+l)$ of length $S(d,k,k+l)$ then $T(C)$ is isomorphic to the transition sequence (\ref{eq8}).  Hence all such circuit codes are isomorphic.
\end{proof}

\section{Conclusions and Conjectures}
In this article (assuming the ``usual'' conditions on $d,k, \text{and } l$, i.e. as articulated in Section 3) we used the conditions of symmetry and the existence of a long bit run to prove two results: (i) the formula $S(d,k,k+l)=4k+2l$ holds for even values of $k$ (extending Theorem \ref{byrnes1}), and (ii) when $l=2 \text{ or } 3$ all maximum length symmetric circuit codes in $\mathcal{F}(d,k,k+l)$ are isomorphic (extending Theorem 4 of \cite{Douglas2}).
Taken together, these two results (combined in Theorem \ref{byrnes2} parts (i) and (ii)) establish that circuit codes satisfying these conditions (i.e. the symmetric members of $\mathcal{F}(d,k,k+l))$ represent one of the few families of non-trivial circuit codes that are well-understood.

The contribution of Theorem \ref{byrnes2} (i) is clear, however it may not immediately be obvious how Theorem \ref{byrnes2} (ii) extends \cite{Douglas2} Theorem 4, as the latter does not explicitly require symmetry or a long bit run.
Recall that \cite{Douglas2} Theorem 4 states that when $k$ is odd and $d=\frac{1}{2}(3k+3)$, all $(d,k)$ circuit codes of length $K(d,k)=4k+4$ are isomorphic.
Let $C^*$ denote one such maximum length circuit code.
For $d=\frac{1}{2}(3k+3)$, the value $l=2$ satisfies the equation $d=\frac{1}{2}(3k+l+1)$.
And since $|C^*|=K(d,k)=4k+4>2(k+1)$, it follows from \cite{Singleton} Theorem 1 that $T(C^*)$ must possess a bit run of length $k+2=k+l$.
Hence $C^*\in \mathcal{F}(d,k,k+l)$ for $l=2$.
Furthermore, it is evident in the proof of \cite{Douglas2} Theorem 4 that for $k$ odd and $d=\frac{1}{2}(3k+3)$ all maximum length $(d,k)$ circuit codes are symmetric.
This means that the assumptions of Theorem \ref{byrnes2} are in fact implicitly satisfied for any odd $k\ge 5$, and in that sense \cite{Douglas2} Theorem 4 equates to the $l=2$ case of Theorem \ref{byrnes2} part (ii).
(This also indicates from where our explicit conditions in Theorem \ref{byrnes2} came.)

We remark that the analysis required for the $l=3$ case, while owing a great debt to the analysis of \cite{Douglas2}, is substantially more complex than the $l=2$ case, as can be seen from the fact that Lemma \ref{lem3_2} is proven for $l=2$ in the first paragraph.  The type of argument used in the concluding part of that proof (beginning in paragraph 4) appears novel in a circuit code context and may be of independent interest.

Before concluding, we investigate whether part (ii) of Theorem \ref{byrnes2} can be strengthened by increasing the range of $l$ (e.g. is it valid $\forall l\ge 2$?).
In Example \ref{ex3} we show that the result is not true in general for odd $l\ge 3$ by exhibiting non-isomorphic symmetric circuit codes in $\mathcal{F}(15,8,13)$ both having length $S(15,8,13)$.

\begin{example}
\label{ex3}
Consider the symmetric $(15,8,13)$ circuit codes $C$ and $C'$ with transition sequences $T(C)=(\omega_A,\omega_A)$ and $T(C')=(\omega_B,\omega_B)$ where:
\small
\[
\omega_A=(\underbrace{1,2,3,4,5,6,7,8,9,10,11,12,13,2,4,6,8}_{\omega},14,10,15,12)
\]
\[
\omega_B=(1,2,3,4,5,6,7,8,9,10,11,12,13,2,6,4,8,14,10,15,12).
\]
\normalsize
Notice $|C|=|C'|=42=S(15,8,13)$, but $C$ is not isomorphic to $C'$ as no segment of $T(C')$ is isomorphic to $\omega$.
Here the implied value of $l$ is 5, and isomorphism fails as assumption (\ref{eq7}) is violated.
In particular, since $l=5$, the second paragraph of the proof of Lemma \ref{lem3_2} need not hold, which allows for these non-isomorphic circuit codes.
\end{example}

Finally, we recall Conjecture 3.12 of \cite{Byrnes2019} which postulates (under the same conditions on $d,k, \text{and } l$ as in Section 3) that $S(d,k,k+l)=K(d,k)$ for $k$ odd and $l$ even.
When $k$ is odd and $l=2 \text{ or } 4$ this equality is confirmed by the formulas for $K(d,k)$ from \cite{Douglas2} Theorems 4 and 5.
Here we have extended the formula $S(d,k,k+l)=4k+2l$ to the case where $k$ is even and $l$ odd.
In particular, when $l=3$ and $k$ is even $\ge 4$, Theorem \ref{byrnes2} states 
$S(\frac{3k}{2},k,k+3)=4k+6$, coinciding with the value of $K(\frac{3k}{2},k)$ (for even $k$) from \cite{Douglas2} Theorem 3.
Further supporting evidence is provided by the best-known lower bound on $K(15,8)$ of 42 as reported in \cite{Byrnes}.
For $d=15$ and $k=8$, the value $l=5$ satisfies the formula $d=\frac{1}{2}(3k+l+1)$, and since $k=8\ge 2l-2$ we can apply Theorem \ref{byrnes2} (i) to get $S(15,8,13)=42$.
Hence we strengthen the conjecture to:

\begin{conj}
\label{conj1}
Let $k$ and $l$ be integers $\ge 2$ of opposite parity with $k\ge 2l+1$ if $k$ odd, $k\ge 2l-2$ if $k$ even, and $d=\frac{1}{2}(3k+l+1)$.
Then $S(d,k,k+l)=K(d,k)$.
\end{conj}

\footnotesize
\bibliographystyle{plain}
\bibliography{CC_Max_Length_and_Iso_Bib}
\normalsize

\end{document}